\documentclass[12pt]{article}

\usepackage{amsmath,amsfonts,amssymb,amsthm,cite}
\usepackage{geometry}

\theoremstyle{plain}
\newtheorem{theorem}{Theorem}
\newtheorem{lemma}{Lemma}

\newtheorem{definition}{Definition}
\theoremstyle{definition}

\newtheorem{remark}{Remark}

\newcommand{\nonprint}[1]{}

\begin{document}

\begin{flushleft}

\small
DOI 10.1007/s11253-019-01599-7

\textit{Ukrainian Mathematical Journal, Vol.70, No.11, April, 2019 (Ukrainian Original Vol.70, No.11, November, 2018)}
\vspace{+0.2cm}

\textbf{O.\,M.~Atlasiuk, V.\,A.~Mikhailets} \small(Institute of Mathematics of NAS of Ukraine, Kyiv)

\Large

\textbf{Fredholm one-dimensional boundary-value problems \\with parameter in Sobolev spaces}
\end{flushleft}

\normalsize

\begin{abstract}
For systems of linear differential equations on a compact interval, we investigate the~dependence on a parameter $\varepsilon$ of the solutions to boundary-value problems in the Sobolev spaces $W^{n}_{\infty}$. We~obtain a constructive criterion of the continuous dependence of the solutions of these problems on the~parameter $\varepsilon$ for $\varepsilon=0$. The degree of convergence of these solutions is established.
\end{abstract}

\section{Introduction}\label{section1}

The investigation of solutions of the systems of ordinary differential equations is an important part of numerous problems of contemporary analysis and its applications (see, e.g.,~\cite{BochSAM2004} and the references therein). For linear boundary-value problems, the conditions for the Fredholm property and continuous dependence of solutions on the parameters were established by Kiguradze \cite{Kigyradze1975, Kigyradze1987}. Later, his results were generalized by the second author of the present paper and his colleagues \cite{KodliukMikhailets2013JMS, MPR2018, MikhailetsChekhanova}. Recently, these investigations were extended to more general classes of Fredholm boundary-value problems in various Banach function spaces \cite{GKM2015, KodlyukM2013, GKM2017, MMS2016, AtlMikh2018}. These problems have a series of specific features and require the application of new approaches and methods.

\section{Statement of the problem}\label{section2}

 Consider a finite interval $(a,b)\subset\mathbb{R}$ and given numbers
$$
\{m, n\} \subset \mathbb{N}, \quad \varepsilon_0>0.
$$
  We study a family of inhomogeneous boundary-value problems of the form
 \begin{equation}\label{equation_e}
 L(\varepsilon)y(t;\varepsilon):= y'(t;\varepsilon) + A(t;\varepsilon) y(t;\varepsilon) = f(t;\varepsilon), \quad t\in (a,b),
\end{equation}
\begin{equation}\label{bound_cond_e}
B(\varepsilon)y(\cdot;\varepsilon) = c(\varepsilon),
\end{equation}
 parametrized by a number $\varepsilon \in [0,\varepsilon_0)$. Here, for any fixed value of the parameter $\varepsilon$, the matrix function $$A(\cdot;\varepsilon)  \in W_\infty^{n-1}\bigl([a,b];\mathbb{C}^{m\times m}\bigr)=:\bigl(W^{n-1}_\infty\bigr)^{m\times m},$$ the vector function $$f(\cdot;\varepsilon) \in W_\infty^{n-1}\bigl([a,b];\mathbb{C}^{m}\bigr)=:\bigl(W^{n-1}_\infty\bigr)^m,$$ the vector $c(\varepsilon) \in \mathbb{C}^m$, and $B(\varepsilon)$ is a linear continuous operator
$$
B(\varepsilon) \colon \bigl(W^{n}_\infty \bigr)^m\rightarrow \mathbb{C}^m.
$$

A solution of the boundary-value problem \eqref{equation_e}, \eqref{bound_cond_e} is defined as a vector function $y(\cdot;\varepsilon) \in (W^{n}_\infty)^m$ satisfying equation \eqref{equation_e}
almost everywhere on $(a,b)$ (everywhere for $n\geq 1$) and equality~\eqref{bound_cond_e}. The boundary condition \eqref{bound_cond_e} is the most general condition for system \eqref{equation_e} whose solution runs over the entire Sobolev space $(W^{n}_\infty)^m$ \cite[Lemma 1]{AtlMikh2018}. The boundary-value problem \eqref{equation_e}, \eqref{bound_cond_e} can be associated with the linear operator
\begin{equation}\label{opLB}
(L(\varepsilon),B(\varepsilon)) \colon \bigl(W^{n}_\infty\big)^m\to \bigr(W^{n-1}_\infty\big)^m\times\mathbb{C}^{m}.
\end{equation}
This is a Fredholm operator with index zero \cite[Theorem 1]{AtlMikh2018}.

The main aim of the present paper is to establish a criterion for the continuous dependence of the solutions of boundary-value problems of the form \eqref{equation_e}, \eqref{bound_cond_e} on the parameter $\varepsilon$ for $\varepsilon=0$.

\section{Main results}\label{section3}

We now formulate the main results of the present paper. Their proof is presented in Section~4.

In order that the analyzed problem be meaningful, in what follows, we assume than \textbf{condition (0)} is satisfied, namely, a \it  boundary-value problem of the form \eqref{equation_e}, \eqref{bound_cond_e} {\samepage
\begin{equation*}
L(0)y(t;0)=0,\quad t\in (a,b),\quad B(0)y(\cdot;0)=0
\end{equation*}
possesses solely the trivial solution.}\rm

In this case, the corresponding limiting inhomogeneous boundary-value problem possesses a unique solution.

Consider the following \textbf{boundary conditions} as $\varepsilon\to0+$:
\begin{itemize}
  \item [(I)] $A(\cdot;\varepsilon)\to A(\cdot;0)$ in the space $\big(W^{n-1}_\infty\big)^{m\times m}$;
  \item [(II)] $B(\varepsilon)y\to B(0)y$ in $\mathbb{C}^{m}$ for any $y\in\big(W^{n}_\infty\big)^m$.
\end{itemize}

\begin{definition}\label{defin_3} We say that a solution of the boundary-value problem \eqref{equation_e}, \eqref{bound_cond_e} continuously depends on the parameter $\varepsilon$ for $\varepsilon=0$ if the following conditions are satisfied:
\begin{itemize}
\item[$(\ast)$] there exists a positive number $\varepsilon_{1}<\varepsilon_{0}$ such that, for any $\varepsilon\in[0,\varepsilon_{1})$, arbitrary right-hand sides $f(\cdot;\varepsilon)\in \big(W^{n-1}_\infty\big)^{m}$, and  $c(\varepsilon)\in\mathbb{C}^{m}$, this problem has a unique solution $y(\cdot;\varepsilon)$ that belongs to the space $\big(W^{n}_\infty\big)^{m}$;
\item [$(\ast\ast)$] the convergence of the right-hand sides $f(\cdot;\varepsilon)\to f(\cdot;0)$ in $\big(W_\infty^{n-1}\big)^{m}$ and $c(\varepsilon)\to c(0)$ in $\mathbb{C}^{m}$ implies the convergence of the solutions
\begin{equation*}\label{4.gu}
y(\cdot;\varepsilon)\to y(\cdot;0)\quad\mbox{in}\quad\big(W^{n}_\infty\big)^{m}
\quad\mbox{as}\quad\varepsilon\to0+.
\end{equation*}
\end{itemize}
\end{definition}

We now formulate a criterion for the continuity of the solution $ y = y(t,\varepsilon)$ of the boundary-value problem \eqref{equation_e},~\eqref{bound_cond_e} with respect to the parameter $\varepsilon$ as $\varepsilon\rightarrow 0+$ in the space $W^{n}_\infty$.

\begin{theorem}\label{3_5}
A solution of the boundary-value problem \eqref{equation_e}, \eqref{bound_cond_e} continuously depends on the parameter~$\varepsilon$ for $\varepsilon=0$ if and only if it satisfies condition \textup{(0)} and the boundary conditions~\textup{(I)} and \textup{(II)}.
\end{theorem}

We proceed to the investigation of the rate of convergence of solutions to the boundary-value problem \eqref{equation_e}, \eqref{bound_cond_e} as $\varepsilon\to0+$.

We set
\begin{equation*}
\widetilde{d}_{n-1,\infty}(\varepsilon):=
\bigl\|L(\varepsilon)y(\cdot;0)-f(\cdot;\varepsilon)\bigr\|_{n-1,\infty}+
\bigl\|B(\varepsilon)y(\cdot;0)-c(\varepsilon)\bigr\|_{\mathbb{C}^{m}},
\end{equation*}
where $\|\cdot\|_{n-1,\infty}$ is the norm in the space $W^{n-1}_\infty$ and $\|\cdot\|_{\mathbb{C}^{m}}$ is the norm in the space ${\mathbb{C}^{m}}$.

The quantities
$$
\bigl\|y(\cdot;0)-y(\cdot;\varepsilon)\bigr\|_{n,\infty}
$$
and $\widetilde{d}_{n-1,\infty}(\varepsilon)$ are, respectively, the error and discrepancy of the solution $y(\cdot;\varepsilon)$ of the boundary-value problem \eqref{equation_e}, \eqref{bound_cond_e} if $y(\cdot;\varepsilon)$ is its exact solution and $y(\cdot;0)$ is an approximate solution of the~problem.

\begin{theorem}\label{3.6.th-bound}
Suppose that the boundary-value problem \eqref{equation_e}, \eqref{bound_cond_e} satisfies conditions \textup{(0)}, \textup{(I)}, and \textup{(II)}. Then there exist positive quantities $\varepsilon_{2}<\varepsilon_{1}$ and $\gamma_{1}$, $\gamma_{2}$ such that, for any $\varepsilon\in(0,\varepsilon_{2})$, the~following two-sided estimate is true:
\begin{equation}\label{3.6.bound}
\begin{aligned}
\gamma_{1}\,\widetilde{d}_{n-1,\infty}(\varepsilon)
\leq\bigl\|y(\cdot;0)-y(\cdot;\varepsilon)\bigr\|_{n,\infty}\leq
\gamma_{2}\,\widetilde{d}_{n-1,\infty}(\varepsilon),
\end{aligned}
\end{equation}
where the quantities $\varepsilon_{2}$, $\gamma_{1}$ and $\gamma_{2}$ are independent $y(\cdot;0)$ and $y(\cdot;\varepsilon)$.
\end{theorem}

By virtue of this theorem, the error and discrepancy of the solution $y(\cdot;\varepsilon)$ of the boundary-value problem \eqref{equation_e}, \eqref{bound_cond_e} have the same order of smallness.

\section{Auxiliary results}\label{section4}

 The theorem presented below contains constructive conditions under which the continuous operator \eqref{opLB} is
invertible for sufficiently small values of the parameter $\varepsilon$ and guarantees the~continuous dependence of solutions
on the parameter in the space $(W^n_\infty)^{m}$.

\begin{theorem}\label{th_sol_conv}
Suppose that the following conditions are satisfied as $\varepsilon\rightarrow 0+$:

\begin{enumerate}
\renewcommand{\theenumi}{\arabic{enumi})}
\renewcommand{\labelenumi}{\arabic{enumi})}
\item\label{th_epsilon_cond_1-bis} $\|A(\cdot;\varepsilon) - A(\cdot;0)\|_{n-1,\infty}\rightarrow 0$ in the space $\left(W^{n-1}_\infty\right)^{m\times m}$;
\item\label{th_epsilon_cond_2-bis-bis} $B(\varepsilon)y\rightarrow B(0)y$ for any $ y\in \left(W^{n}_\infty\right)^m$.
\end{enumerate}
Then, for sufficiently small $\varepsilon>0$,
the operator $\left(L(\varepsilon),B(\varepsilon)\right)$ is invertible. In addition, if
\begin{enumerate}
\renewcommand{\theenumi}{\arabic{enumi})}
\renewcommand{\labelenumi}{\arabic{enumi})}
\setcounter{enumi}{2}
\item\label{th_epsilon_cond_3-bis} $\|f(\cdot;\varepsilon) - f(\cdot;0)\|_{n-1,\infty}\rightarrow 0$ and  $c(\varepsilon)\rightarrow c(0)$,
\end{enumerate}
then the solution $y(\cdot,\varepsilon)$
of problem \eqref{equation_e}, \eqref{bound_cond_e} satisfies the limit property \begin{equation}\label{solution_converge_1}
\left\|y(\cdot;\varepsilon)-y(\cdot;0)\right\|_{n,\infty}\rightarrow 0.
\end{equation}
\end{theorem}

We present the proof of Theorem \ref{th_sol_conv} in the form of four lemmas formulated in what follows:
\begin{lemma}\label{L,B}
 Suppose that condition (0) and conditions \ref{th_epsilon_cond_1-bis} and \ref{th_epsilon_cond_2-bis-bis} of Theorem \ref{th_sol_conv} are satisfied. Then, for sufficiently small $\varepsilon>0$, the operator $(L(\varepsilon),B(\varepsilon))$ is invertible.
\end{lemma}

\begin{proof} Under condition \ref{th_epsilon_cond_1-bis}, by the theorem on homeomorphisms from \cite{AtlMikh2018}, we get
\begin{equation}\label{homeomor}
\|Y(\cdot;\varepsilon)-Y(\cdot;0)\|_{n,\infty}\rightarrow 0, \quad \varepsilon\rightarrow 0+.
\end{equation}
Thus, by using condition \ref{th_epsilon_cond_2-bis-bis}, we establish the convergence of numerical matrices:
\begin{equation}\label{BY_converge}
\left[B(\varepsilon)Y(\cdot;\varepsilon)\right] \rightarrow
\left[B(0)Y(\cdot;0)\right], \quad \varepsilon\rightarrow 0+.
\end{equation}
According to condition $(0)$, the limit square matrix is nonsingular \cite[Theorem 2]{AtlMikh2018}.
Hence, for sufficiently
small $\varepsilon\geq 0$, we find
\begin{equation*}
\det
\left[B(\varepsilon)Y(\cdot;\varepsilon)\right]\neq 0.
\end{equation*}
This yields the invertibility of the operator $(L(\varepsilon),B(\varepsilon))$.

\end{proof}

Parallel with the original inhomogeneous boundary-value problem \eqref{equation_e}, \eqref{bound_cond_e} for the vector function $y(t;\varepsilon)$, we consider the following three vector boundary-value problems:
 \begin{equation}\label{v}
    v'(t;\varepsilon)= - A(t;\varepsilon)v(t;\varepsilon),\quad
 B(\varepsilon) v(\cdot;\varepsilon) =c(\varepsilon),
\end{equation}
 \begin{equation}\label{uKoshi}
  x'(t;\varepsilon) + A(t;\varepsilon)x(t;\varepsilon)=f(t;\varepsilon),\quad x(a;\varepsilon)= 0,
 \end{equation}
\begin{equation*}\label{w}
 w'(t;\varepsilon) + A(t;\varepsilon)w(t;\varepsilon)=f(t;\varepsilon),\quad B(\varepsilon) w(\cdot;\varepsilon) = 0,
\end{equation*}
where the parameter $\varepsilon\geq0$ is small. It is known that the boundary-value (Cauchy) problem~\eqref{uKoshi} is uniquely solvable.

By using Lemma \ref{uKoshi}, we arrive at the equality
\begin{equation}\label{spivvidnoshennya*}
 y(\cdot;\varepsilon)=v(\cdot;\varepsilon)+w(\cdot;\varepsilon)
\end{equation}
for small $\varepsilon\geq0$. Hence, in order to prove Theorem \ref{th_sol_conv}, it suffices to show that, under its conditions, the following relations hold as $\varepsilon\rightarrow 0+$:
\begin{equation}\label{v_1,1}
\|v(\cdot;\varepsilon)-v(\cdot;0)\|_{n,\infty} \rightarrow 0,
\end{equation}
\begin{equation}\label{w_1,1}
\|w(\cdot;\varepsilon)-w(\cdot;0)\|_{n,\infty} \rightarrow 0.
\end{equation}

\begin{lemma}\label{v_e}
Suppose that the conditions of Theorem \ref{th_sol_conv} are satisfied as $\varepsilon\rightarrow 0+$. Then the limit relation \eqref{v_1,1} is true.
\end{lemma}

\begin{proof} By using the first equality in the boundary-value problem \eqref{v}, we obtain
\begin{equation}\label{vv}
v(\cdot;\varepsilon)=Y(\cdot;\varepsilon) \widetilde{c}(\varepsilon)
\end{equation}
for some $\widetilde{c}(\varepsilon) \in \mathbb{C}^{m}$. In view of the second equality in problem \eqref{v}, we find
 $$[B(\varepsilon) Y(\cdot;\varepsilon)]\widetilde{c}(\varepsilon) =c(\varepsilon).$$
By virtue of Lemma \ref{L,B}, the criterion of invertibility from \cite[Theorem 2]{AtlMikh2018}, relation \eqref{BY_converge}, and condition~\ref{th_epsilon_cond_2-bis-bis}, we get
$$
\widetilde{c}(\varepsilon)=[B(\varepsilon) Y(\cdot;\varepsilon)]^{-1}c(\varepsilon) \rightarrow [B(0) Y(\cdot;0)]^{-1}c(0)=\widetilde{c}(0), \quad \varepsilon\rightarrow 0+.
$$
Relation \eqref{v_1,1} is derived from \eqref{homeomor} and \eqref{vv}.

\end{proof}

\begin{lemma}\label{Koshi}
Suppose that conditions \ref{th_epsilon_cond_1-bis}--\ref{th_epsilon_cond_3-bis} of Theorem \ref{th_sol_conv} are satisfied as $\varepsilon\rightarrow 0+$ Then the~solution of problem~\eqref{uKoshi} has the following property:
\begin{equation}\label{Koshi e}
 \|x(\cdot;\varepsilon)-x(\cdot;0)\|_{n,\infty}\rightarrow 0,\quad \varepsilon\rightarrow 0+.
\end{equation}
 \end{lemma}

\begin{proof} Assume that the number $\varepsilon>0$ is sufficiently small. The solution of problem \eqref{uKoshi} admits the following representation:
\begin{equation}\label{x}
x(t;\varepsilon)=Y^{-1}(t;\varepsilon)\int\limits_a^t Y(s;\varepsilon)f(s;\varepsilon){\rm d}s.
\end{equation}
Under condition \ref{th_epsilon_cond_1-bis}, by the theorem on homeomorphisms in \cite{AtlMikh2018}, we get
\begin{equation}\label{Y+-}
\bigl\|Y^{\pm1}(\cdot;\varepsilon)-Y^{\pm1}(\cdot;0)\bigr\|_{n,\infty}\rightarrow 0
\end{equation}
as $\varepsilon\rightarrow 0+$. According to condition \ref{th_epsilon_cond_3-bis} and relation \eqref{Y+-}, we find
\begin{equation}\label{Yf}
\|Y(\cdot;\varepsilon)f(\cdot;\varepsilon)-Y(\cdot;0)f(\cdot;0)\|_{n-1,\infty}\rightarrow 0
\end{equation}
because $W^{n}_\infty$ is a Banach algebra. Thus, relation \eqref{Koshi e} follows from relations \eqref{x}--\eqref{Yf}.

\end{proof}

\begin{lemma}\label{vor w}
Under the conditions of Theorem \ref{th_sol_conv}, the limit relation \eqref{w_1,1} is true.
\end{lemma}

\begin{proof} The vector function $u(\cdot;\varepsilon)=x(\cdot;\varepsilon)-w(\cdot;\varepsilon)$ is a solution of a boundary-value problem of the form \eqref{v}:
\begin{equation*}\label{u1}
\begin{array}{c}
u'(t;\varepsilon)=-A(t;\varepsilon)u(t;\varepsilon),  \\
 B(\varepsilon) u(\cdot;\varepsilon) = B(\varepsilon)x(\cdot;\varepsilon) =: \widetilde{c}(\varepsilon).
 \end{array} \end{equation*}
By using property \ref{th_epsilon_cond_2-bis-bis} and Lemma \ref{Koshi}, we get $\widetilde{c}(\varepsilon)\rightarrow \widetilde{c}(0)$ as $\varepsilon\rightarrow 0+$.
 It follows from Lemma \ref{v_e} that
 \begin{equation}\label{u_1,1}
 \|u(\cdot;\varepsilon)-u(\cdot;0)\|_{n,\infty} \rightarrow 0,\quad \varepsilon\rightarrow 0+.
 \end{equation}
In view of the equality $w(\cdot;\varepsilon)=x(\cdot;\varepsilon)- u(\cdot;\varepsilon)$ and relations (\ref{Koshi e}) and (\ref{u_1,1}), we obtain (\ref{w_1,1}).

The required limit property \eqref{solution_converge_1} is a direct corollary of equality \eqref{spivvidnoshennya*} and Lemmas~ \ref{v_e} and~\ref{vor w}.

\end{proof}

Theorem \ref{th_sol_conv} is proved.

\begin{remark}\label{eqiv def}
Definition \ref{defin_3} is equivalent to the following definition:
\end{remark}

\begin{definition}\label{defin_3'} We say that a solution of the boundary-value problem \eqref{equation_e}, \eqref{bound_cond_e} continuously depends on the parameter $\varepsilon$ for $\varepsilon=0$ if the following conditions are satisfied:
\begin{itemize}
\item[$(\ast)$] there exists a positive number $\varepsilon_{1}<\varepsilon_{0}$ such that, for any $\varepsilon\in[0,\varepsilon_{1})$, arbitrary right-hand sides $f(\cdot)\in \left(W^{n-1}_\infty\right)^{m}$, and  $c\in\mathbb{C}^{m}$, this problem possesses a unique solution $y(\cdot;\varepsilon)\in\left(W^{n}_\infty\right)^{m}$;
\item [$(\ast\ast)$] the following limit relation for the convergence of solutions is true:
\begin{equation*}\label{4.gu'}
y(\cdot;\varepsilon)\to y(\cdot;0)\quad\mbox{in}\quad\left(W^{n}_\infty\right)^{m}
\quad\mbox{as}\quad\varepsilon\to0+.
\end{equation*}
\end{itemize}
\end{definition}

The conditions of Definition \ref{defin_3'} directly follow from Definition \ref{defin_3}. We now prove the converse implication.

By Theorem \ref{th_sol_conv}, the operator  $\bigl(L(\varepsilon),B(\varepsilon)\bigr)$ has a bounded inverse operator $$\bigl(L(\varepsilon),B(\varepsilon)\bigr)^{-1}\colon\left(W^{n-1}_\infty\right)^m\times\mathbb{C}^{m} \to \left(W^{n}_\infty\right)^m$$ for any $\varepsilon \in [0,\varepsilon_2')$. Moreover, by Definition \ref{defin_3'}, for sufficiently small $\varepsilon$, we get the following strong convergence of inverse operators:
\begin{equation}\label{zb ob op}
(L(\varepsilon),B(\varepsilon))^{-1} \xrightarrow{s} (L(0),B(0))^{-1},
\end{equation}
as well as the convergence of right-hand sides:
\begin{equation}\label{zb fc}
f(\cdot;\varepsilon)\rightarrow f(\cdot;0), \quad c(\varepsilon)\rightarrow c(0).
\end{equation}

We choose $f(\cdot;\varepsilon) \in \left(W^{n-1}_\infty\right)^{m}$ and $c(\varepsilon) \in \mathbb{C}^{m}$. Then the equalities
\begin{equation}\label{riv ob}
y(\cdot;\varepsilon)=(L(\varepsilon),B(\varepsilon))^{-1}(f(\cdot;\varepsilon),c(\varepsilon)),
\end{equation}
\begin{equation}\label{riv ob1}
y(\cdot;0)=(L(0),B(0))^{-1}(f(\cdot;0),c(0)),
\end{equation}
are true, i.e., the convergence of $y(\cdot;\varepsilon)$ to $y(\cdot;0)$ is equivalent to the convergence
\begin{equation}\label{zb ob}
(L(\varepsilon),B(\varepsilon))^{-1} (f(\cdot;\varepsilon),c(\varepsilon)) \rightarrow  (L(0),B(0))^{-1}(f(\cdot;\varepsilon),c(\varepsilon)), \quad \varepsilon\to0+.
\end{equation}
By the Banach–-Steinhaus theorem, for sufficiently small $\varepsilon$, we obtain
\begin{equation}\label{2.1.4}
\quad \bigl\|(L(\varepsilon),B(\varepsilon))^{-1} \bigr\|\leqslant C.
\end{equation}
Since
$$
\bigl\|(L(\varepsilon),B(\varepsilon))^{-1}(f(\cdot;\varepsilon),c(\varepsilon))-(L(0),B(0))^{-1}(f(\cdot;0),c(0))\bigr\|\leqslant$$
$$\leqslant\bigl\|(L(\varepsilon),B(\varepsilon))^{-1}\bigr\|\bigl\|(f(\cdot;\varepsilon),c(\varepsilon))-(f(\cdot;0),c(0))\bigr\| +$$ $${}+\bigl\|[(L(\varepsilon),B(\varepsilon))^{-1}-(L(0),B(0))^{-1}](f(\cdot;0),c(0))\bigr\|,
$$
by using conditions \eqref{zb ob op}--\eqref{2.1.4}, we prove that the limit relation $(\ast\ast)$ in Definition \ref{defin_3} is true.

We now establish one more auxiliary result. Assume that the operator $(L(0),B(0))$ is invertible. We now consider the following two conditions:
\begin{itemize}
\item[$(i)$] the operator $(L(\varepsilon),B(\varepsilon))^{-1}$ converges to the operator $(L(0),B(0))^{-1}$ in the strong operator topology;
\item [$(ii)$] the operator $(L(\varepsilon),B(\varepsilon))$ converges to the operator $(L(0),B(0))$ in the strong operator topology.
\end{itemize}

\begin{theorem}\label{2.dop}
Conditions $(i)$ and $(ii)$ are equivalent, i.e., as $\varepsilon\to0+$,
\begin{equation}\label{operatoru}
(L(\varepsilon),B(\varepsilon))^{-1}\xrightarrow{s}(L(0),B(0))^{-1}\Longleftrightarrow(L(\varepsilon),B(\varepsilon))\xrightarrow{s}(L(0),B(0)).
\end{equation}
\end{theorem}

We split the procedure of substantiation of equivalence \eqref{operatoru} into two steps.

\emph{Step 1.} We prove that the strong convergence of inverse operators is equivalent to the set of conditions (I) and (II).

\emph{Step 2.} We prove that the set of conditions (I) and (II) is equivalent to the strong convergence of operators.

Indeed, by Theorem \ref{3_5}, the operator $(L(\varepsilon),B(\varepsilon))$ has a bounded inverse operator $(L(\varepsilon),B(\varepsilon))^{-1}$. Moreover, we also have strong convergence of the inverse operators. Thus, Theorem \ref{3_5} immediately implies the validity of conditions of Step 1. Note that, for any continuous operators acting in infinite-dimensional Banach spaces and depending on $\varepsilon$, this equivalence is not true.

We now proceed to Step 2. To prove this assertion, we show that the following lemma is true:

\begin{lemma}\label{2.l.zb.l}
The boundary condition \textup{(I)} is equivalent to each of the following conditions:
\begin{itemize}
\item[$(a_1)$] $\|L(\varepsilon)- L(0)\|\to0$ as $\varepsilon\to0+$;
\item [$(a_2)$] $L(\varepsilon)y \to L(0)y$ in $\left(W^{n-1}_\infty\right)^{m}$ as $\varepsilon\to0+$ for every $y\in\left(W^{n}_\infty\right)^{m}$.
\end{itemize}
\end{lemma}

\begin{proof} The implication $(a_1) \Rightarrow (a_2)$ is obvious. It remains to show that condition $(a_1)$ follows from the
boundary condition (I) and the boundary condition (I) follows from condition $(a_2)$. We first substantiate the first
implication. Assume that $\|A(\varepsilon)- A(0)\|_{n-1,\infty}\to0$ as $\varepsilon\to0+$. For any vector function $y\in\left(W^{n}_\infty\right)^{m}$, we obtain
\begin{gather*}
\|(L(\varepsilon)-L(0))y\|_{n-1,\infty} = \|(A(\varepsilon)-A(0))y\|_{n-1,\infty} \leq \\
\leq c_{n-1,\infty}\|A(\varepsilon)-A(0)\|_{n-1,\infty}\|y\|_{n-1,\infty}\leq\\ \leq c_{n}\|A(\varepsilon)-A(0)\|_{n-1,\infty}\|y\|_{n,\infty}
\quad\mbox{as}\quad\varepsilon\to0+.
\end{gather*}

Here, $c_{n}$ is a positive number independent of $y$. This number exists because $W^{n}_\infty$ is a Banach algebra. Hence,
\begin{equation*}
\|L(\varepsilon)-L(0)\|\leq c_{n}\|A(\varepsilon)-A(0)\|_{n-1,\infty}\to0
\quad\mbox{as}\quad\varepsilon\to0+,
\end{equation*}

where $\|\cdot\|$ denotes the norm of a linear continuous operator on the pair of spaces
\begin{equation*}
L(\varepsilon) \colon \left(W^{n}_\infty\right)^{m}\rightarrow \left(W^{n-1}_\infty\right)^{m}.
\end{equation*}

The first implication is proved.

We now prove that the boundary condition (I) follows from condition$(a_2)$. Assume that condition $(a_2)$ is
satisfied. Then
$$
Y'+A(\varepsilon)Y=[L(\varepsilon)Y]\to[L(0)Y]=Y'+A(0)Y \quad\mbox{in}\quad
\left(W^{n-1}_\infty\right)^{m\times m}
$$
as $\varepsilon\to0+$ for any matrix function $Y\in\left(W^{n}_\infty\right)^{m\times m}$. Moreover, the matrix function  $[L(\varepsilon)Y]$ is formed by columns obtained are a result of the action of the operator $L(\varepsilon)$ on the corresponding columns of the matrix $Y$. Setting $Y(t)\equiv I_{m}$, we arrive at the required convergence of $A(\varepsilon)\to A(0)$ in $\left(W^{n-1}_\infty\right)^{m\times m}$ as $\varepsilon\to0+$. The second implication and, hence, Lemma \ref{2.l.zb.l} are proved. \end{proof} Theorem \ref{2.dop} is proved.

\section{Proofs of Theorems \ref{3_5} and~\ref{3.6.th-bound}}\label{section5}
\begin{proof}[Proof of Theorem \ref{3_5}]
The \textit{sufficiency} of conditions (0), (I), and (II) under which problem \eqref{equation_e},~\eqref{bound_cond_e} satisfies Definition \ref{defin_3} is established in Theorem \ref{th_sol_conv}. We prove the necessity. Assume that problem \eqref{equation_e},~\eqref{bound_cond_e} satisfies Definition \ref{defin_3}. Then condition (0) is valid. It remains to show that this problem satisfies conditions (I) and (II). We split the proof into three steps:

\emph{Step 1.} We prove that the boundary-value problem \eqref{equation_e}, \eqref{bound_cond_e} satisfies the boundary condition~(I). Under the condition $(\ast)$ of Definition \ref{defin_3}, the operator
\begin{equation*}\label{oper_e}
\bigl(L(\varepsilon),B(\varepsilon)\bigr) \colon \left(W^{n}_\infty\right)^{m}\rightarrow \left(W^{n-1}_\infty\right)^{m}\times\mathbb{C}^{m}
\end{equation*}
is invertible for any $\varepsilon\in[0,\varepsilon_{1})$. For any $\varepsilon\in[0,\varepsilon_{1})$, we consider a matrix boundary-value problem
 $$
 Y'(t;\varepsilon)+A(t;\varepsilon)Y(t;\varepsilon)=O_{m}, \quad t\in (a,b),
 $$
 $$
  [BY(\cdot;\varepsilon)]=I_{m}.
 $$
  This boundary-value problem is a collection of $m$ boundary-value problems \eqref{equation_e}, \eqref{bound_cond_e} with right-hand sides independent of $\varepsilon$. By assumption, this problem is uniquely solvable and its solution $Y(\cdot;\varepsilon)\in\left(W^{n}_\infty\right)^{m\times m}$ satisfies the condition $Y(\cdot;\varepsilon)\to Y(\cdot;0)$ in the space $\left(W^{n}_\infty\right)^{m\times m}$ as $\varepsilon\to0+$. Note that $\det Y(t;\varepsilon)\neq0$ for any $t\in(a,b)$
because otherwise the columns of the matrix function  $Y(\cdot;\varepsilon)$ are linearly dependent, which contradicts the condition $[BY(\cdot;\varepsilon)]=I_{m}$. Hence,
  $$
   A(\cdot;\varepsilon)=-Y'(\cdot;\varepsilon)(Y(\cdot;\varepsilon))^{-1} \rightarrow -Y'(\cdot;0)(Y(\cdot;0))^{-1}=A(\cdot;0)
  $$
  in the space $\left(W^{n}_\infty\right)^{m\times m}$ as $\varepsilon\to0+$, i.e., condition (I) is satisfied.

\emph{Step 2.} We now show that condition (II) is satisfied. First, we prove that $\|B(\varepsilon)\|=O(1)$ as $\varepsilon\to0+$, where $\|\cdot\|$ is the norm of the bounded operator $B(\varepsilon) \colon (W^{n}_\infty)^m\to\mathbb{C}^{m}$. Assume the contrary, i.e., that there exists a number sequence $\left(\varepsilon^{(k)}\right)_{k=1}^{\infty}\subset(0,\varepsilon_{1})$ such that $\varepsilon^{(k)}\to0$ and
 $$
0<\left\|B\bigl(\varepsilon^{(k)}\bigr)\right\|\to\infty, \quad \varepsilon\to0+.
 $$
 For each number $k$, we choose a vector function
$x_{k}\in(W^{n}_\infty)^{m}$ such that
$$
\|x_{k}\|_{n,\infty}=1 \quad \mbox{and} \quad \bigl\|B\bigl(\varepsilon^{(k)}\bigr)x_{k}\bigr\|_{\mathbb{C}^{m}}\geq \frac{1}{2}
\bigl\|B\bigl(\varepsilon^{(k)}\bigr)\bigr\|.
$$

We set
\begin{gather*}
y\bigl(\cdot;\varepsilon^{(k)}\bigr):=
\bigl\|B\bigl(\varepsilon^{(k)}\bigr)\bigr\|^{-1}x_{k},\\
f\bigl(\cdot;\varepsilon^{(k)}\bigr):=
L\bigl(\varepsilon^{(k)}\bigr)\,y\bigl(\cdot;\varepsilon^{(k)}\bigr),\\
c\bigl(\varepsilon^{(k)}\bigr):=B\bigl(\varepsilon^{(k)}\bigr)\,y\bigl(\cdot;\varepsilon^{(k)}\bigr). \end{gather*}

Since $y\left(\cdot;\varepsilon^{(k)}\right)\to0$ in the space $\left(W^{n}_\infty\right)^{m}$ as $\varepsilon\to0+$, we have $f\left(\cdot;\varepsilon^{(k)}\right)\to0$ in $\left(W^{n-1}_\infty\right)^{m}$ because it has
already been shown that $A(\cdot;\varepsilon)$ satisfies condition (I). Since the finite-dimensional space $\mathbb{C}^{m}$ is locally compact, the inequalities $$1/2\leq\left\|c\left(\varepsilon^{(k)}\right)\right\|_{\mathbb{C}^{m}}\leq1$$ are true. Passing to a subsequence of numbers $\varepsilon^{(k)}$, we can assume that $c\left(\varepsilon^{(k)}\right)\to c(0)$ as $k\to\infty$, where $c(0)$ is a nonzero vector in $\mathbb{C}^{m}$. Thus, for any number $k$, the vector function $y\left(\cdot;\varepsilon^{(k)}\right)\in\left(W^{n}_\infty\right)^{m}$ is a unique solution of the boundary-value problem
\begin{gather*}
L\left(\varepsilon^{(k)}\right)\,y\left(t;\varepsilon^{(k)}\right)=f\left(t;\varepsilon^{(k)}\right),  \quad t\in (a,b),\\
B\left(\varepsilon^{(k)}\right)\,y\left(\cdot;\varepsilon^{(k)}\right)=c\left(\varepsilon^{(k)}\right).
\end{gather*}

Recall that $f\left(\cdot;\varepsilon^{(k)}\right)\to0$ in $\left(W^{n-1}_\infty\right)^{m}$ and $c\left(\varepsilon^{(k)}\right)\to c(0)\neq0$ in $\mathbb{C}^{m}$ as $k\to\infty$. Under the condition $(\ast\ast)$ of Definition \ref{defin_3}, the function $y\left(\cdot;\varepsilon^{(k)}\right)$ converges in the space $\left(W^{n}_\infty\right)^{m}$ to the unique solution $y(\cdot;0)$ of the limit boundary-value problem formed by the differential equation $L(0)y(t,0)=0$, $t\in (a,b)$, and the inhomogeneous
boundary condition $B(0)y(\cdot;0)=c(0)$. Since $y\left(\cdot;\varepsilon^{(k)}\right)\to0$ in the same space, we conclude that $y(\cdot;0)\equiv0$, which contradicts the boundary condition. Therefore, this assumption is not true, i.e., $\|B(\varepsilon)\|=O(1)$ as $\varepsilon\to0+$.

\emph{Step 3.} We now show that condition (II) is satisfied. It follows from the result established above that there
exist numbers $\gamma'>0$ and $\varepsilon'\in(0,\varepsilon_{1})$ such that $\|(L(\varepsilon),B(\varepsilon))\|\leq\gamma'$ for all $\varepsilon\in[0,\varepsilon')$, where $\|\cdot\|$ is the norm of a bounded operator acting from the space $\left(W^{n}_\infty\right)^{m}$ into the space $\left(W^{n-1}_\infty\right)^{m}\times\mathbb{C}^{m}$. We arbitrarily choose
a vector function $y\in\left(W^{n}_\infty\right)^{m}$ and set $f(\cdot;\varepsilon):=L(\varepsilon)y$ and  $c(\varepsilon):=B(\varepsilon)y$ for any $\varepsilon\in[0,\varepsilon_{0})$. Hence, in view of condition $(\ast\ast)$, as $\varepsilon\to0+$, we get
\begin{gather*}
\bigl\|B(\varepsilon)y-B(0)y\bigr\|_{\mathbb{C}^{m}}
\leq\bigl\|(f(\cdot;\varepsilon),c(\varepsilon))-
(f(\cdot;0),c(0))\bigr\|_{\left(W^{n-1}_\infty\right)^{m}\times\mathbb{C}^{m}}=\\
=\bigl\|(L(\varepsilon),B(\varepsilon))(L(\varepsilon),B(\varepsilon))^{-1}(f(\cdot;\varepsilon),c(\varepsilon))-
(f(\cdot;0),c(0))\bigr\|_{\left(W^{n-1}_\infty\right)^{m}\times\mathbb{C}^{m}}\leq\\
\leq\gamma'\,\bigl\|(L(\varepsilon),B(\varepsilon))^{-1}
\bigl((f(\cdot;\varepsilon),c(\varepsilon))-
(f(\cdot;0),c(0))\bigr)\bigr\|_{n,\infty}=\\
=\gamma'\,\bigl\|(L(0),B(0))^{-1}(f(\cdot;0),c(0))
-(L(\varepsilon),B(\varepsilon))^{-1}(f(\cdot;0),c(0))\bigr\|_{n,\infty}
\to0.
\end{gather*}

Since $$\|B(\varepsilon)\|=O(1) \quad \mbox{and} \quad \bigl\|B(\varepsilon)y-B(0)y\bigr\|_{\mathbb{C}^{m}} \to0,$$ we conclude that
 $B(\varepsilon)y$ converges to $B(0)y$ in $\mathbb{C}^{m}$ for any $y\in\left(W^{n}_\infty\right)^m$.
Thus, the boundary-value problem \eqref{equation_e}, \eqref{bound_cond_e} satisfies condition (II).

\end{proof}

\begin{proof}[Proof of Theorem~\ref{3.6.th-bound}]
First, we show that the left-hand side of the two-sided inequality \eqref{3.6.bound} holds. We set
\begin{equation}\label{fc}
f(\cdot,\varepsilon):=
L(\varepsilon)\,y(\cdot;\varepsilon),\quad
c(\varepsilon):=B(\varepsilon)\,y(\cdot;\varepsilon).
\end{equation}
The strong convergence of inverse operators
$$
\bigl(L(\varepsilon),B(\varepsilon)\bigr)\xrightarrow{s} \bigl(L(0),B(0)\bigr), \quad \varepsilon\to0+,
$$
follows from the boundary conditions (I) and (II). Hence, there exist numbers $\gamma'>0$ and $\varepsilon \in (0,\varepsilon_2')$ such that the norm of this operator satisfies the inequality
\begin{equation}\label{2.1.4`}
\quad \left\|(L(\varepsilon),B(\varepsilon)) \right\|\leqslant \gamma'.
\end{equation}
Indeed, if we assume the contrary, then we can find a sequence of positive numbers $\bigl(\varepsilon^{(k)}\bigr)_{k=1}^{\infty}$ such that $$\varepsilon^{(k)}\to0 \quad \mbox{and} \quad \bigl\|\bigl(L(\varepsilon^{(k)}),B(\varepsilon^{(k)})\bigr)\bigr\|\rightarrow\infty \quad \mbox{as} \quad k\rightarrow\infty.$$ However, by the Banach–-Steinhaus theorem, this contradicts the fact that $\bigl(L(\varepsilon^{(k)}),B(\varepsilon^{(k)})\bigr)$ strongly converge to $(L(0),B(0))$ as \smash{$k\rightarrow\infty$}. By using \eqref{fc} and \eqref{2.1.4`}, for any $\varepsilon \in (0,\varepsilon_2')$, we conclude that
\begin{gather*}
\bigl\|L(\varepsilon)y(\cdot;0)-f(\cdot;\varepsilon)\bigr\|_{n-1,\infty}+
\bigl\|B(\varepsilon)y(\cdot;0)-c(\varepsilon)\bigr\|_{\mathbb{C}^{m}}=\\
=\bigl\|L(\varepsilon)y(\cdot;0)-L(\varepsilon)y(\cdot;\varepsilon)\bigr\|_{n-1,\infty}+
\bigl\|B(\varepsilon)y(\cdot;0)-B(\varepsilon)y(\cdot;\varepsilon)\bigr\|_{\mathbb{C}^{m}}\leq\\
\leq\bigl\|L(\varepsilon)\bigr\|\bigl\|y(\cdot;0)-y(\cdot;\varepsilon)\bigr\|_{n,\infty}+
\bigl\|B(\varepsilon)\bigr\|\bigl\|y(\cdot;0)-y(\cdot;\varepsilon)\bigr\|_{n,\infty}
\leq\gamma'\bigl\|y(\cdot;0)-y(\cdot;\varepsilon)\bigr\|_{n,\infty}.
\end{gather*}
Thus, we have established the left-hand side of inequality \eqref{3.6.bound} with $\gamma_{1}:=1/\gamma'$.

We now prove the right-hand side of the two-sided inequality \eqref{3.6.bound}. By Theorem \ref{3_5}, the operator $(L(\varepsilon),B(\varepsilon))$ has a bounded inverse operator $(L(\varepsilon),B(\varepsilon))^{-1}$ for any $\varepsilon \in (0,\varepsilon_2')$. Moreover, we have the following strong convergence:
$$
(L(\varepsilon),B(\varepsilon))^{-1}\xrightarrow{s} (L(0),B(0))^{-1}, \quad \varepsilon\to0+.
$$
Indeed, for any $f \in \bigl(W^{n-1}_\infty\bigr)^{m}$ and $c \in \mathbb{C}^{m}$,  under the condition $(\ast\ast)$ of Definition \ref{defin_3}, we get the following convergence:
$$
\bigl(L(\varepsilon),B(\varepsilon)\bigr)^{-1}(f;c)=:y(\cdot;\varepsilon)\rightarrow y(\cdot;0):= \bigl(L(0),B(0)\bigr)^{-1}(f;c)
$$
in $\bigl(W^{n}_\infty\bigr)^{m}$ as $\varepsilon\to0+$. As above, by the Banach–-Steinhaus theorem, the norms of these inverse operators are bounded, i.e., there exist positive numbers $\varepsilon_2$ and $\gamma_{2}$ such that the norm of the inverse operator
$$
\quad \bigl\|\bigl(L(\varepsilon),B(\varepsilon)\bigr)^{-1} \bigr\|\leqslant\gamma_{2}.
$$
Thus, for any $\varepsilon \in (0,\varepsilon_2)$, the relations
\begin{gather*}
\bigl\|y(\cdot;0)-y(\cdot;\varepsilon)\bigr\|_{n,\infty}
=\bigl\|\bigl(L(\varepsilon),B(\varepsilon)\bigr)^{-1}\bigl(L(\varepsilon),B(\varepsilon)\bigr)\bigl(y(\cdot;0)-
y(\cdot;\varepsilon)\bigr)\bigr\|_{n,\infty}\leq\\
\leq\gamma_{2}\bigl(\,\|L(\varepsilon)y(\cdot;0)-f(\cdot;\varepsilon)\|_{n-1,\infty}+
\|B(\varepsilon)y(\cdot;0)-c(\varepsilon)\|_{\mathbb{C}^{m}}\bigr)
\end{gather*}
are true. This directly yields the right-hand side of the two-sided estimate~\eqref{3.6.bound}.

\end{proof}

\end{document}